\documentclass[11pt]{amsart}

\usepackage[utf8]{inputenc}
\usepackage[english]{babel}
\usepackage [autostyle, english = american]{csquotes}
\MakeOuterQuote{"}
\usepackage[a4paper,twoside,top=1.2in, bottom=1.1in, left=0.9in, right=0.9in]{geometry}
\usepackage{thmtools}

\usepackage{graphicx}
\usepackage{tikz}
\usetikzlibrary{graphs}
\usepackage{caption} % Better handling of empty figure captions.
\usepackage{amsmath}
\usepackage{amsthm}
\usepackage{amssymb}
\usepackage{esint} % For the average integral
\usepackage{mathrsfs} % Just for cool sigma-algebra
\usepackage{xcolor}
\usepackage{array}
\usepackage{hhline}
\usepackage[shortlabels]{enumitem}
\usepackage{comment} % For comment environment
\usepackage[toc,page]{appendix} % Appendices
\usepackage{xparse} % In order to use ExplSyntax (in older latex versions)
\usepackage{mathtools}
\usepackage{fancyhdr} % Header and footer of pages
\usepackage{ifthen} % Ifthen construct
\usepackage{forloop} % forloop construct
\usepackage{xstring}
\usepackage{emptypage} % Leave completely empty the empty pages
\usepackage{setspace}
\usepackage[initials,alphabetic]{amsrefs} % For alphanumeric labels: alphabetic
\usepackage{tikz}
\usetikzlibrary{arrows,arrows.meta,shapes,patterns,calc,fadings,decorations.pathreplacing,decorations.markings,decorations.pathmorphing,backgrounds}

\usepackage{hyperref} % References become hyperlinks.
\hypersetup{
	colorlinks = true,
	linkcolor = {blue},
	urlcolor = {red},
	citecolor = {blue}
}

\usepackage[nameinlink,capitalise]{cleveref} % References show what kind of thing is being referenced (theorem, lemma, proposition...)

\newcounter{results}[section] % Uniform counters for lemmas, theorems, propositions etc

\theoremstyle{plain}
\newtheorem{theorem}[results]{Theorem}
\newtheorem{lemma}[results]{Lemma}
\newtheorem{proposition}[results]{Proposition}

\newtheorem*{theorem*}{Theorem}
\newtheorem*{lemma*}{Lemma}
\newtheorem*{proposition*}{Proposition}
\newtheorem*{corollary*}{Corollary}
\newtheorem*{exercise*}{Exercise}
\newtheorem*{fact*}{Fact}
\newtheorem*{claim*}{Claim}
\newtheorem*{observation*}{Observation}
\newtheorem*{question*}{Open question}

\theoremstyle{remark}
\newtheorem{remark}[results]{Remark}

\newtheorem*{remark*}{Remark}

\theoremstyle{definition}
\newtheorem{definition}[results]{Definition}

\newtheorem*{definition*}{Definition}
\newtheorem*{example*}{Example}

\numberwithin{equation}{section}

\crefname{figure}{Figure}{Figures}

\ifdefined\comma 
        \renewcommand{\comma}{\ensuremath{\, \text{, }}}
\else 
        \newcommand{\comma}{\ensuremath{\, \text{, }}}
\fi

\newcommand{\N}{\ensuremath{\mathbb N}}%Natural numbers
\newcommand{\R}{\ensuremath{\mathbb R}}%Real numbers
\DeclarePairedDelimiter\abs{\lvert}{\rvert} % Absolute value
\newcommand{\sk}[1]{\ensuremath{\langle #1 \rangle}} % Scalar product 

\newcommand{\eqdef}{\ensuremath{\coloneqq}} % Equal in a definition.
\DeclareMathOperator{\diam}{diam}

\DeclareMathOperator{\Ric}{Ric} % Ricci tensor
\newcommand{\amb}{\ensuremath{M}} % Ambient manifold
\DeclareMathOperator{\Graph}{Graph} % Graph of an embedded surface
\colorlet{myGray}{gray}
\colorlet{myBlue}{blue}
\colorlet{myBlack}{black}
\colorlet{myBackground}{gray!10}

\title[Unknottedness of FBMS and self-shrinkers]{Unknottedness of \\ free boundary minimal surfaces and self-shrinkers}
\author{Sabine Chu and Giada Franz}
\newcommand\printaddress{
\setlength{\parindent}{17pt}
\footnotesize

\bigskip
\par 
{\scshape \noindent Sabine Chu}
\newline MIT, Department of Mathematics, Cambridge, MA 02139, USA.
\newline
\textit{E-mail address:} \texttt{srchu@mit.edu}
\newline
\par
{\scshape \noindent Giada Franz}
\newline Université Gustave Eiffel, CNRS, LAMA, 77420 Champs-sur-Marne, France.
\newline
\textit{E-mail address:} \texttt{giada.franz@cnrs.fr}
\newline
\par
}

\begin{document}
\begin{abstract}
    We study unknottedness for free boundary minimal surfaces in a three-dimensional Riemannian manifold with nonnegative Ricci curvature and strictly convex boundary, and for self-shrinkers in the three-dimensional Euclidean space.
    For doing so, we introduce the concepts of boundary graph for free boundary minimal surfaces and of graph at infinity for self-shrinkers. We prove that these surfaces are unknotted in the sense that any two such surfaces with isomorphic boundary graph or graph at infinity are smoothly isotopic.
\end{abstract}

\maketitle

\section{Introduction}\label{sec:intro}

Let $(M^3,g)$ be a three-dimensional Riemannian manifold, possibly noncompact or with nonempty boundary. Let $\Sigma$ be a smooth surface that is properly embedded in $\amb$, possibly with boundary $\partial\Sigma=\Sigma\cap\partial \amb$.  Recall that $\Sigma$ is a \textit{minimal surface} if its mean curvature $H$ is equal to zero. 
Moreover, if $\Sigma$ is minimal and meets $\partial \amb$ orthogonally, we call it a \textit{free boundary minimal surface}.
Equivalently, $\Sigma$ is a critical point of the area functional with respect to variations constraining the boundary of the surface to the boundary $\partial \amb$. 

Lawson in \cite{Lawson1970unknot} proved that minimal surfaces in the three-dimensional sphere $S^3$ are unknotted, building on \cite{Waldhausen1968}. Here, following Lawson's definition, we say that a surface is \emph{unknotted} if there exists an ambient isotopy mapping the surface to a fixed standardly embedded surface with the same topology (see \cref{sec:prelim} for more precise definition).
It is then natural to pose the same question for different ambient manifolds $\amb$. First, note that Lawson's proof works for minimal surfaces in simply connected three-dimensional manifolds with positive Ricci curvature (which are diffeomorphic to $S^3$ by result of Hamilton \cite{Hamilton1982}*{Theorem~1.1}).
Moreover, Meeks--Yau showed in \cite{MeeksYau1992}*{Theorem~5.1} that complete minimal surfaces in $\R^3$ with finite topological type are unknotted. 
The result was then extended to remove the assumption of finite topological type in \cite{FrohmanMeeks1997} and \cite{FrohmanMeeks2008}.
Note that some positivity assumption on the curvature of the ambient manifold is needed to prove unknottedness of minimal surfaces. Indeed, there are examples of knotted minimal surfaces in the hyperbolic space $\mathbb H^3$ by \cite{deOliveiraSoret1998}*{Theorem~2} (see also Remark~2 after the theorem).

Here, we consider two settings that are a priori quite different but actually display interesting similarities. Namely, we are interested in:
\begin{itemize}
\item Free boundary minimal surfaces in the unit ball $B^3\subset\R^3$, or more in general in a compact three-dimensional Riemannian manifold $(\amb^3,g)$ with nonnegative Ricci curvature and strictly convex boundary. Note that such a manifold $\amb$ is diffeomorphic to the unit ball $B^3$ by \cite{FraserLi2014}*{Theorem~2.11}.
\item Self-shrinkers with finite topology in $\R^3$. Recall that these are minimal surfaces in $\R^3$ with respect to the Gaussian metric $e^{-\abs{x}^2/4}g_{\R^3}$ and they appear as singularity models of the mean curvature flow (see \cite{Huisken1990}*{Theorem~3.5}).
\end{itemize}

Unfortunately, the problem in these settings turns out to be more complicated.
Consider for example the free boundary case. Note that in the unit ball $B^3\subset\R^3$, there exist free boundary minimal surfaces of the same genus and number of boundary components for which there is no ambient diffeomorphism mapping one into the other. For example, the surfaces with three boundary components (and large genus) obtained in \cite{KapouleasLi2017}*{Theorem~1.1} (see also \cite{Ketover2016FBMS}*{Theorem~1.1}) versus \cite{KarpukhinKusnerMcGrathStern2024}*{Theorem~1.2} are not ambient diffeomorphic. Note that Kapouleas and Li construct their surfaces by desingularizing the union of the critical catenoid and the equatorial disc, while the surfaces obtained by Karpukhin--Kusner--McGrath--Stern are doublings of the equatorial disc. This means that these constructions differ in the arrangement of the boundary components on the sphere, as displayed in \cref{fig:three-comp}.

A similar behavior is expected from self-shrinkers, for which the ends play the role of the boundary components.
Observe that Wang in \cite{Wang2016}*{Theorem~1.1} proved a structure theorem for the ends of a self-shrinker in $\R^3$. However, this is not as strong as the result for minimal surfaces in $\R^3$, for which the geometric arrangement of the ends is always the same (see e.g. \cite{MeeksYau1992}*{Theorem~4.1}).

These considerations suggest that any notion of unknottedness should take into account the arrangement of the boundary components of free boundary minimal surfaces and of the ends of self-shrinkers.
To this purpose, following \cite{Frohman1992}, we introduce the concept of boundary graph for free boundary minimal surfaces and of graph at infinity for self-shrinkers and we proved that, for any two such surfaces with the same genus and isomorphic graph, there exists an ambient diffeomorphism mapping one into the other.
In the next two sections we discuss in more detail the two settings separately.

\subsection{Free boundary minimal surfaces}
In this section, let us assume that $(M^3,g)$ is a compact Riemannian manifold with nonnegative Ricci curvature and nonempty strictly convex boundary.
The examples of free boundary minimal surfaces discussed above suggest the concept of \emph{boundary graph} $\Graph(\partial\Sigma,\partial M)$ for a properly embedded surface $\Sigma \subset M$. Essentially, we let the vertex set be the set of connected components of $\partial M\setminus\partial\Sigma$, and connect vertices whose corresponding components are both adjacent to some component of $\partial\Sigma$. See \cref{def:boundary-graph} for a more detailed definition. For instance, the boundary graphs for the surfaces in \cite{KapouleasLi2017} and \cite{KarpukhinKusnerMcGrathStern2024} are shown in the last column of \cref{fig:three-comp}. In particular, observe that these two graphs are not isomorphic, therefore there cannot exist an ambient diffeomorphism mapping one surface into the other.

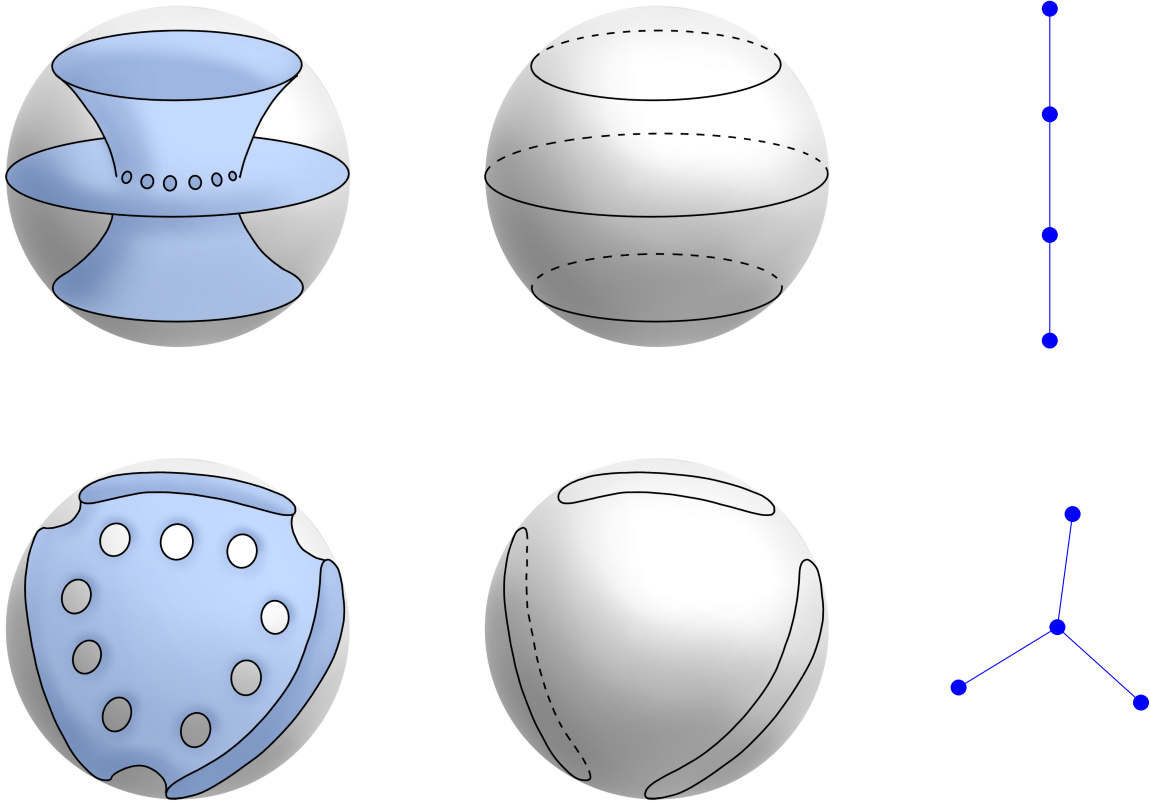
\begin{figure}[htbp]
\centering
\begin{tikzpicture}
\node[inner sep=0pt] (catdisc) at (0,0){\includegraphics[width=.28\textwidth]{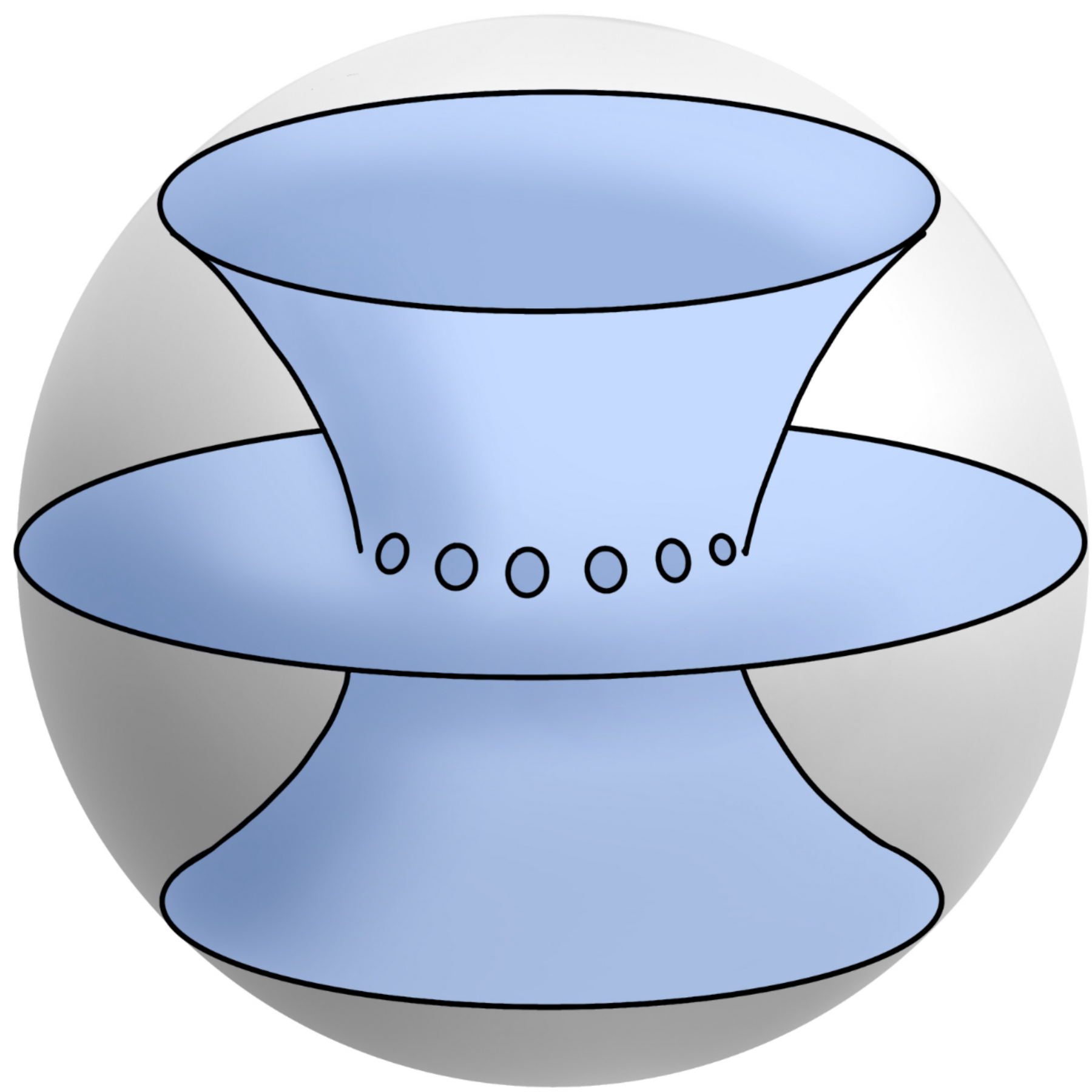}};
\node[inner sep=0pt] (ddisc) at (0,-5.4)
    {\includegraphics[width=.28\textwidth]{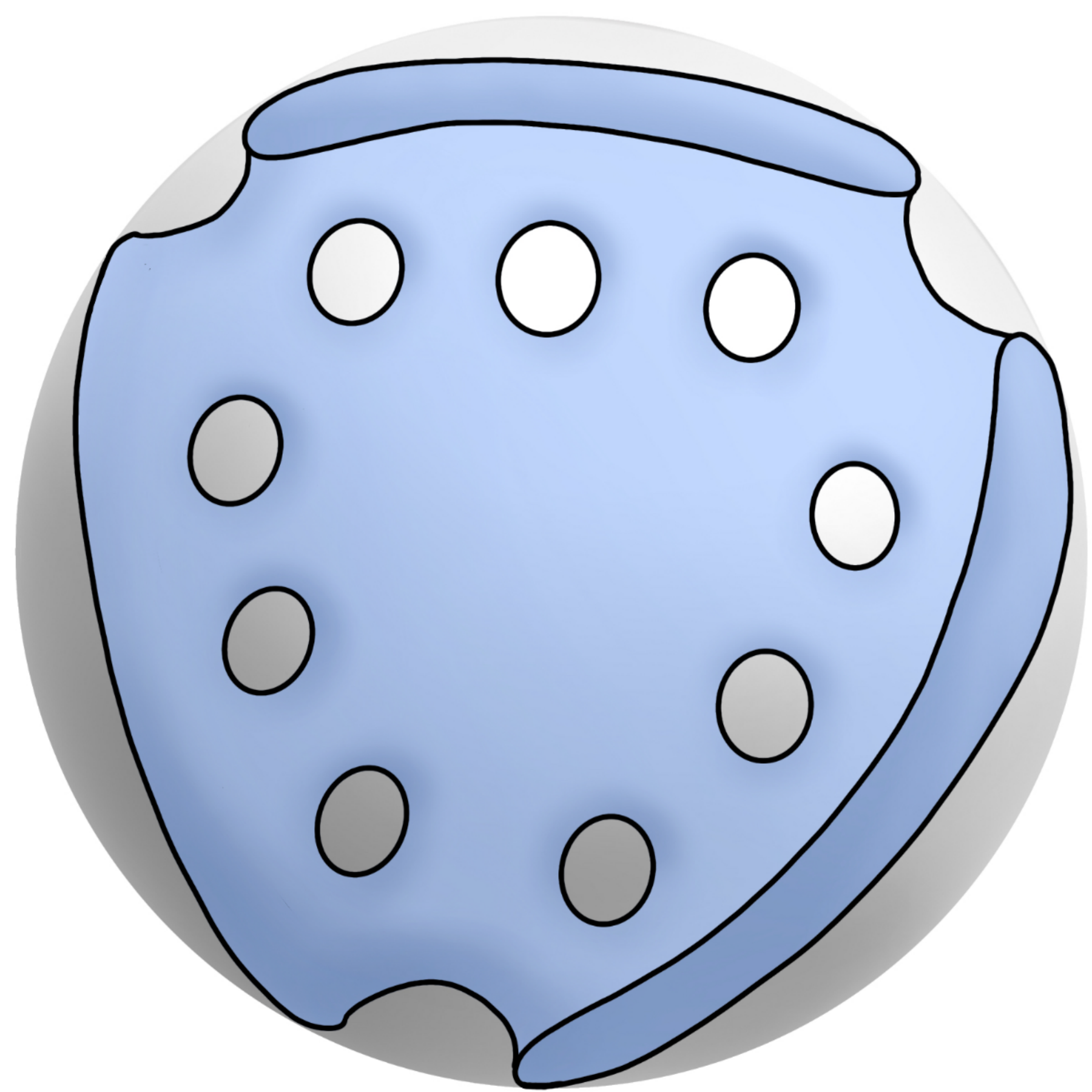}};
\node[inner sep=0pt] (catdiscbdry) at (6.3,0){\includegraphics[width=.28\textwidth]{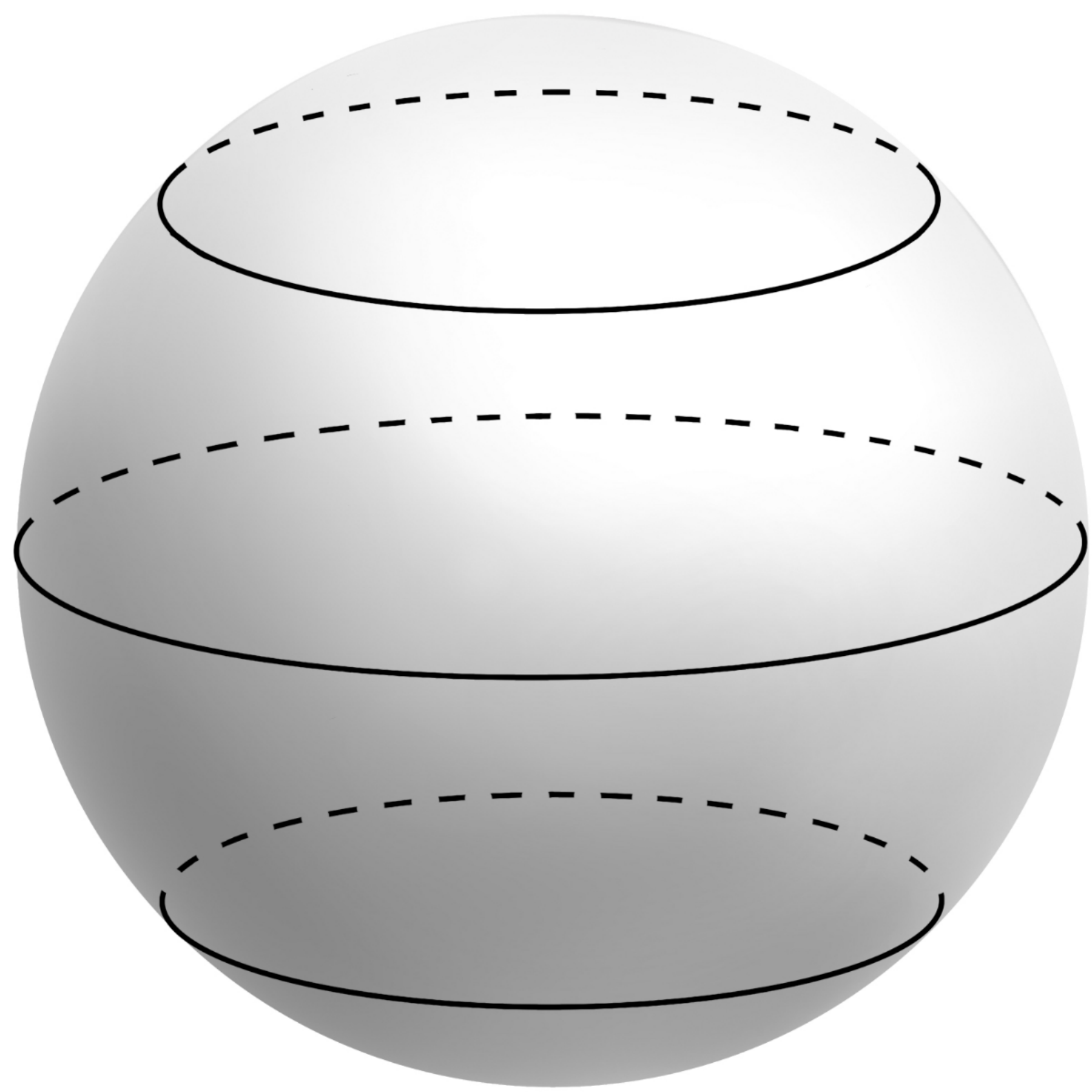}};
\node[inner sep=0pt] (ddiscbdry) at (6.3,-5.4)
    {\includegraphics[width=.28\textwidth]{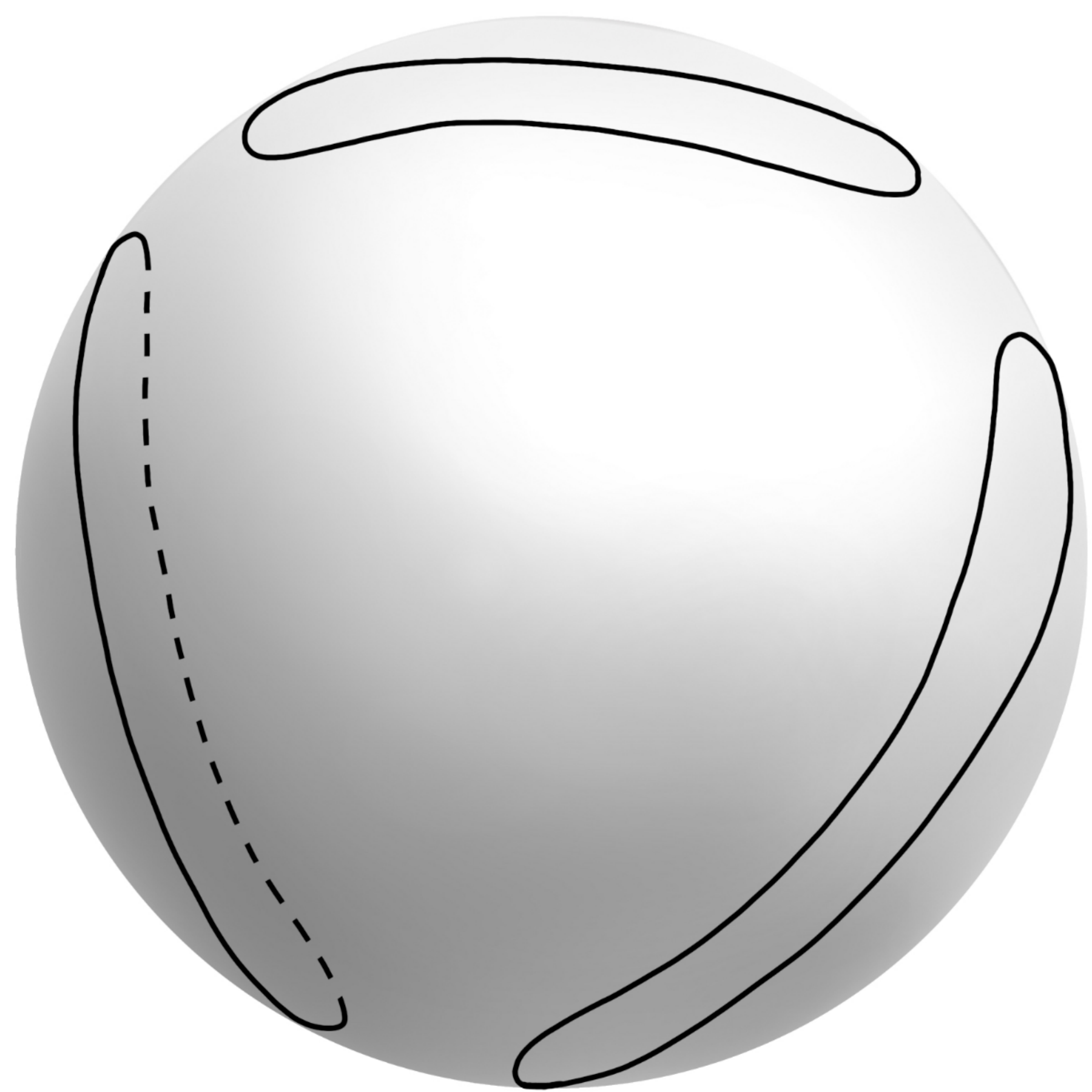}};

\coordinate (a) at (11.5,2.2);
\coordinate (b) at (11.5,0.8);
\coordinate (c) at (11.5,-0.8);
\coordinate (d) at (11.5,-2.2);
\draw[color=blue] (a) -- (b) -- (c) -- (d);
\foreach \x in {(a), (b), (c), (d)}{
        \fill[blue] \x circle[radius=3pt];
    }

\begin{scope}[yshift=5mm]
\coordinate (a) at (11.6,-6);
\coordinate (b) at (11.8,-4.5);
\coordinate (c) at (10.3,-6.8);
\coordinate (d) at (12.7,-7);
\draw[blue] (a) -- (b);
\draw[blue] (a) -- (c);
\draw[blue] (a) -- (d);
\end{scope}

\foreach \x in {(a), (b), (c), (d)}{
        \fill[blue] \x circle[radius=3pt];
    }
\end{tikzpicture}
\caption{Schematic representation of the free boundary minimal surfaces in \cite{KapouleasLi2017} (top row) and \cite{KarpukhinKusnerMcGrathStern2024} (bottom row), together with the arrangement of the boundary components on the sphere and the resulting boundary graphs. See \cite{SchulzGallery} for more precise numerical simulations of these surfaces.}
\label{fig:three-comp}
\end{figure}

However, given the notion of boundary graph, we are still able to prove that free boundary minimal surfaces in $M$ are unknotted, in the following sense.

\begin{restatable}{theorem}{freeboundary}\label{thm:freeboundary}
    Let $\Sigma,\ \Sigma'$ be (smooth, compact, properly embedded) free boundary minimal surfaces in a compact three-dimensional Riemannian manifold $(M^3,g)$ with nonnegative Ricci curvature and strictly convex boundary. Assume that $\Sigma$ and $\Sigma'$ have the same genus and that their boundary graphs $\Graph(\partial\Sigma,\partial M)$ and $\Graph(\partial\Sigma',\partial M)$ are isomorphic. Then, $\Sigma$ and $\Sigma'$ are smoothly isotopic in the ambient manifold $M^3$.
\end{restatable}

In the case when $\Sigma$ and $\Sigma'$ have connected boundary or when they are topological annuli, \cref{thm:freeboundary} follows from \cite{Meeks1981}*{Corollary 4}. There, Meeks proves that all minimal surfaces in $B^3$ with connected boundary are unknotted, and that all minimal surfaces in $B^3$ diffeomorphic to an annulus are isotopic to the critical catenoid.
More in general, a notion of unknottedness of free boundary minimal surfaces in this setting was already proven in \cite{ChoeFraser2018}*{Corollary~3.3}. However, Choe--Fraser's result does not contain a discussion on the boundary graphs.
Part of the proof is similar, but we decided to add the entire proof for completeness, together with the discussion on the boundary graphs.

Note that the free boundary condition in \cref{thm:freeboundary} is necessary. Hall in \cite{Hall1984} showed that, if $\Sigma$ and $\Sigma'$ are two minimal surfaces of the same genus in the ball $B^3\subset\R^3$ with the \emph{same boundary}, they are not necessarily isotopic. In particular, he constructs $\Sigma, \Sigma'$ with genus zero and equal boundary where $\Sigma$ is unknotted and $\Sigma'$ is knotted.
The key feature of free boundary minimal surfaces that enables their unknottedness is the Frankel property, as discussed in \cref{sec:KeyIdeas}. Minimal surfaces in $B^3$ without the free boundary condition do not satisfy the Frankel property, which gives a heuristic reason for Hall's examples to exist.

We conclude the discussion about the free boundary setting with an open question, which naturally arises in view of \cref{thm:freeboundary}.
\begin{question*}
Which finite graphs can be realized as boundary graphs of free boundary minimal surfaces in the unit ball $B^3\subset\R^3$ (or more in general in a three-dimensional manifold with nonnegative Ricci curvature and strictly convex boundary)? How does the answer change if we fix the genus of the surface? 
\end{question*}

Recall that Karpukhin--Kusner--McGrath--Stern proved in \cite{KarpukhinKusnerMcGrathStern2024}*{Theorem~1.2} that, for every natural numbers $g\ge 0$, $b\ge 1$, there exists a free boundary minimal surface in the unit ball $B^3$ with genus $g$, $b$ boundary components, and whose boundary graph is a star. Many other examples with different boundary graphs have also been constructed (see the introduction of \cite{FranzKetoverSchulz2024} for a list of recent results).
However, we do not yet have a complete answer to the open question above.
What we can say so far is that boundary graphs of surfaces in this setting have to be trees (see \cref{rmk:trees}).
Moreover, in the special case of a genus zero surface in $B^3$, only very certain graphs can be attained.
Indeed, recall that genus zero free boundary minimal surfaces are radial graphs by \cite{KusnerMcGrath2024}*{Corollary~4.2} (see also \cite{McGrathZou2023}*{Theorem~2.1}. This implies that $\Graph(\partial\Sigma,\partial B^3)$ must be a star.

\subsection{Self-shrinkers}

In this section, our ambient manifold is $(\R^3, e^{-\abs{x}^2/4}g_{\R^2})$. As mentioned above, a minimal surface $\Sigma$ in this space is a self-shrinker. Inspired by the case of free boundary minimal surfaces, we can define the graph at infinity $\Graph(\partial\Sigma,\infty)$ of a self-shrinker $\Sigma$ as $\Graph(\Sigma\cap \partial B_R(0),\partial B_R(0))$ for any radius $R>0$ sufficiently large. Thanks to the structure theorem \cite{Wang2016}*{Theorem~1.1}, this notion of graph at infinity is well-defined (see \cref{sec:selfshrinkers} for more details).

In analogy to the free boundary case, we can prove the following result.

\begin{restatable}{theorem}{selfshrinkers}\label{thm:selfshrinkers}
    Let $\Sigma,\ \Sigma'$ be (smooth, complete, properly embedded) self-shrinkers in $\R^3$ with finite topological type. Assume that $\Sigma$ and $\Sigma'$ have the same genus and that their graphs at infinity $\Graph(\partial\Sigma,\infty)$ and $\Graph(\partial\Sigma',\infty)$ are isomorphic. Then, $\Sigma$ and $\Sigma'$ are smoothly isotopic in $\R^3$.   
\end{restatable}

Note that unknottedness of compact self-shrinkers and self-shrinkers with one or two asymptotically conical ends was proven in \cite{MramorWang2020}*{Theorem~1.1}, \cite{Mramor2024}*{Theorem~1.1}, \cite{Mramor2020}*{Corollary~1.2}.
Here, we generalize these results to any (finite) number of ends.

Similarly to the free boundary case, one can wonder which graphs can arise as graphs at infinity of self-shrinkers.
\begin{question*}
Which graphs can be realized as graphs at infinity of self-shrinkers in $\R^3$? How does the answer change if we fix the genus of the surface?
\end{question*}

Again, very little is known in this direction apart from the genus zero case. In fact, the only genus zero self-shrinkers are the plane, the self-shrinking sphere and the self-shrinking cylinder by \cite{Brendle2016}*{Theorem~2}, for which the graph at infinity has two, one and three vertices, respectively.
This suggests that the self-shrinkers case is probably more rigid than the free boundary one. We refer to the introduction of \cite{Ketover2024} for a discussion on  the known examples of self-shrinkers in $\R^3$.

\subsection{Key ideas of the proofs} \label{sec:KeyIdeas}
Unknottedness of (free boundary) minimal surfaces in an ambient manifold $(M^3,g)$ is closely related to the validity of the Frankel property for minimal surfaces in such manifold.
Indeed, thanks to \cite{Frohman1992}*{Theorem~2.1}, the key step to prove \cref{thm:freeboundary} and \cref{thm:selfshrinkers} is to show that a surface $\Sigma$ as in the statements is a \emph{strong} Heegaard splitting of $M$ in the sense of \cite{Frohman1992}*{Section~1} (here \emph{strong} is used to distinguish this from the standard notion of Heegaard splitting, considered for example in \cite{Meeks1981}). 
Namely, the induced maps $\pi_1(\Sigma)\to\pi_1(C_i)$ are surjective for $i=1,2$, where $C_1$, $C_2$ are the two connected components of $M\setminus\Sigma$. 
It turns out that this is equivalent to proving a Frankel property in the universal Riemannian cover $\tilde C_i$ of $C_i$ for $i=1,2$, which has similar geometric properties as the ambient manifold $M$. 

The Frankel property for free boundary minimal surfaces in a three-dimensional manifold with nonnegative Ricci curvature and strictly convex boundary was proven by Fraser--Li in \cite{FraserLi2014}*{Lemma~2.4}. 
For self-shrinkers in $\R^3$, the Frankel property was proven in increasing generality in \cite{ImperaPigolaRimoldi2021}, \cite{ColdingMinicozzi2023}, and \cite{NaffZhu2024}.
Here we make use of the robust and general statement \cite{NaffZhu2024}*{Proposition~25} to prove the adaptations needed in our setting.
The free boundary setting is treated in \cref{sec:fbms}, while the self-shrinker case is discussed in \cref{sec:selfshrinkers}.
In \cref{sec:prelim}, we prove some topological preliminaries related to \cite{Frohman1992} about strong Heegaard splittings, boundary graphs, and unknottedness.

\subsection*{Acknowledgements} 
We would like to thank Ailana Fraser and Ursula Hamenst\"adt for their interest and support.
Moreover, we would like to thank Keaton Naff for interesting discussions on the Frankel property and on self-shrinkers, Mario Schulz for useful comments on a previous version of the paper, Shengwen Wang for introducing us to the problem and pointing out reference \cite{Meeks1981}, and Jonathan Zung for answering our topological questions. We also thank Robert Kusner and Alex Mramor for their interest on the paper and their useful comments.
Finally, we thank the anonymous referees for the helpful feedback.

G.\,F.~was partially supported by NSF grant DMS-2405361. 
Moreover, part of this work was performed while G.\,F. was in residence at the Simons Laufer Mathematical Sciences Institute (formerly MSRI) during the Fall 2024 semester, supported by NSF grant DMS-1928930.

\section{Topological preliminaries}\label{sec:prelim}

In this section, we collect the topological results about unknottedness needed in the paper.
In particular, we consider a three-dimensional differentiable manifold $M^3$ with boundary. We assume that $M$ is diffeomorphic to the closed unit ball $B^3\subset\R^3$ and we say that $M$ is a \emph{differentiable ball}. We will apply the results in this section to compact Riemannian manifolds with nonnegative Ricci curvature and strictly convex boundary (diffeomorphic to $B^3$ by \cite{FraserLi2014}*{Theorem~2.11}), or to Euclidean balls $B_R(0)\subset\R^3$.
Moreover, unless otherwise stated, we assume that $\Sigma\subset \amb$ is a smooth, compact, connected surface with boundary, properly embedded in $\amb$.

We now start by defining the notion of boundary graph, we then proceed discussing Heegaard splittings and unknottedness.

\begin{definition}[cf. \cite{Frohman1992}*{Section~2}]\label{def:boundary-graph}
    Let $C=\{\gamma_1, \gamma_2, \ldots, \gamma_n\}$ be a set of continuous disjoint simple closed curves lying on a smooth surface $S$. Define the \emph{boundary graph} $\Graph(C,S)$ as follows. The vertices of $\Graph(C,S)$ correspond to the connected components of $S\setminus C$. For every curve $\gamma_i\in C$, an edge connects the vertices corresponding to the connected components that contain $\gamma_i$ in their closure.
\end{definition}
\begin{remark}
Note that the number of edges in $\Graph(C,S)$ is equal to the number of curves in $C$. In particular, it is possible that there are two edges between the same pair of vertices or one edge connecting a vertex to itself. Indeed, if we take the surface $S$ to be a torus and $\gamma_1,\gamma_2$ to be two homotopically nontrivial disjoint curves then:
\begin{itemize}
\item $\Graph(\{\gamma_1,\gamma_2\},S)$ has two vertices and two edges connecting the two vertices;
\item $\Graph(\{\gamma_1\},S)$ has one vertex and one edge connecting the vertex to itself.
\end{itemize}
\end{remark}

\begin{remark} \label{rmk:trees}
Let $M^3$ be a three-dimensional differentiable ball and let $\Sigma\subset \amb$ be a properly embedded surface with $\partial\Sigma=\Sigma\cap\partial \amb$. Then every boundary component of $\partial\Sigma$ divides $\partial \amb$ in two connected components. Therefore, the boundary graph $\Graph(\partial\Sigma,\partial \amb)$ is actually a tree. 

Note also that we can lower bound the number of isomorphism classes of trees with $n$ vertices by $\frac{n^{n-2}}{n!}$ (by Cayley's formula, the number of labelled trees on $n$ vertices is $n^{n-2}$, and there are $n!$ possible labelings), so the questions brought up in the introduction on the possible finite graphs that can be realized as boundary graphs or graphs at infinity is nontrivial.
\end{remark}

\begin{definition}[cf. \cite{Frohman1992}*{Section~1}] \label{def:HeegaardSplitting}
    We say that a (smooth, compact, connected, properly embedded) surface $\Sigma$ in a compact three-dimensional differentiable manifold $M^3$ is a \emph{strong Heegaard splitting} if $M\setminus \Sigma$ consists of two nonempty connected components $C_1$ and $C_2$ (i.e. $\overline{C_1}\cap \overline{C_2} =\Sigma$) such that $C_1,C_2$ are irreducible (i.e., every sphere embedded in $C_i$ bounds a ball), and the induced maps $\pi_1(\Sigma)\to\pi_1(C_i)$ are surjective for $i=1,2$.
\end{definition}

\begin{lemma} \label{lem:HeegaardIrreducibility}
Let $M^3$ be a three-dimensional differentiable ball and let $\Sigma$ be a (smooth, compact, connected, properly embedded) surface in $M$. Then $\Sigma$ is a strong Heegaard splitting of $M$ if and only if the induced maps $\pi_1(\Sigma)\to\pi_1(C_i)$ are surjective for $i=1,2$.
\end{lemma}
\begin{proof}
Observe that, if $M$ is a differentiable ball, then $\Sigma$ is two-sided (by e.g. \cite{ChodoshKetoverMaximo2017}*{Lemma~C.1}) and thus it divides $M$ into two connected components $C_1$ and $C_2$. Moreover, $C_1$ and $C_2$ are irreducible: if $C_1$ was reducible, then there would exist some sphere in $C_1$ whose inside enclosed a portion of $C_2$ (because $M$ is irreducible), but then this would mean there was a connected component of $\Sigma$ inside the sphere that was disjoint from the rest of the surface, which is a contradiction.
\end{proof}

\begin{remark} \label{rmk:HeegaardUniversalCover}
Fix $i=1$ or $i=2$, and let $\Pi\colon\Tilde{C_i}\to C_i$ be the universal cover of $C_i$. If $\Pi^{-1}(\Sigma)$ is connected, then $\pi_1(\Sigma)\to\pi_1(C_i)$ is surjective.
To see this, consider any loop in $C_i$ and lift it to a curve in $\Tilde{C_i}$ connecting two points (which have the same image via $\Pi$). As $\Tilde{C_i}$ is simply connected, this curve is homotopic to a curve in $\Pi^{-1}(\Sigma)$ connecting two points with the same $\Pi$-projection to $\Sigma$ (which exists because $\Pi^{-1}(\Sigma)$ is connected). Then we can project down to a loop on $\Sigma$, so we have surjectivity.
\end{remark}

Note that \cref{def:HeegaardSplitting} is not the classical definition of a Heegaard splitting, which usually does not require the induced maps of the fundamental groups to be surjective. In fact, all minimal surfaces in $M$ are classical Heegaard splittings, regardless of the free boundary condition (see e.g. \cite{Meeks1981}*{Proposition~2}.) Importantly, the $\pi_1$-surjectivity requirement implies that the surface is ``unknotted'' in some way, as we are about to prove.
Indeed, we show that if a surface is a strong Heegaard splitting then it is smoothly isotopic to a paradigmatic unknotted surface, which looks like a thickened unknotted graph and is rigorously defined as follows.

\begin{definition}[Model surface] \label{def:ModelSurface}
Given a tree $T$ and a number $g\in\mathbb{N}$, we define a \emph{model surface} $\Sigma$ in $B^3$ with boundary graph $T$ and genus $g$ as follows (see also \cref{fig:ModelSurface}).

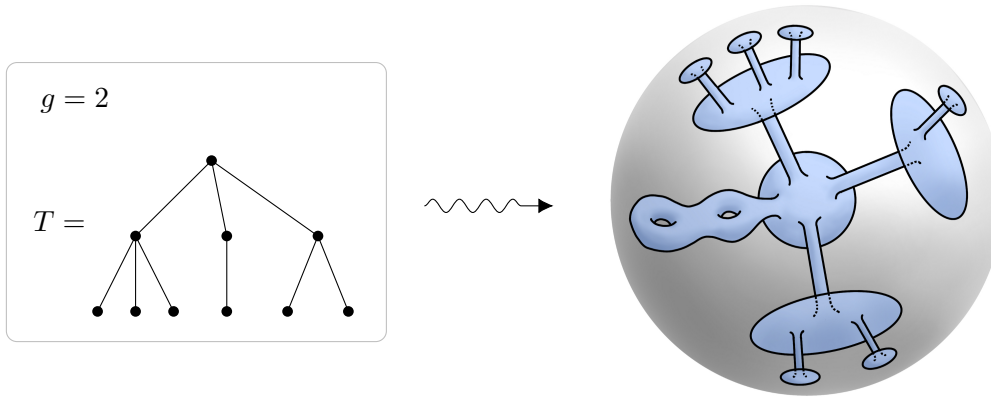
\begin{figure}[htpb]
\centering
\begin{tikzpicture}
\node at (-1.8,0.8) {$g=2$};
\node at (-2,-0.8) {$T=$};
\coordinate (a) at (0,0);
\coordinate (b) at (-1,-1);
\coordinate (c) at (0.2,-1);
\coordinate (d) at (1.4,-1);

\coordinate (e) at (-1.5,-2);
\coordinate (f) at (-1,-2);
\coordinate (g) at (-0.5,-2);

\coordinate (h) at (0.2,-2);

\coordinate (i) at (1,-2);
\coordinate (j) at (1.8,-2);
\draw (a)--(b);
\draw (a)--(c);
\draw (a)--(d);

\draw (b)--(e);
\draw (b)--(f);
\draw (b)--(g);

\draw (c)--(h);

\draw (d)--(i);
\draw (d)--(j);
\foreach \x in {(a), (b), (c), (d), (e), (f), (g), (h), (i), (j)}{
        \fill \x circle[radius=2pt];
    }

\draw[gray!50!white,rounded corners] (-2.7, 1.3) rectangle (2.3,-2.4);

\tikzset{snakearrow/.style={
    decoration={snake,markings,mark=at position 1 with {\arrow[scale=2]{>}}},
    postaction={decorate},
    }
}
\draw[-{Latex[length=2mm,width=2mm]},decorate,decoration={snake,post length=3mm}] (2.8,-0.6) -- (4.5,-0.6);

\node[inner sep=0pt] (modsurf) at (7.8,-0.5){\includegraphics[width=.32\textwidth]{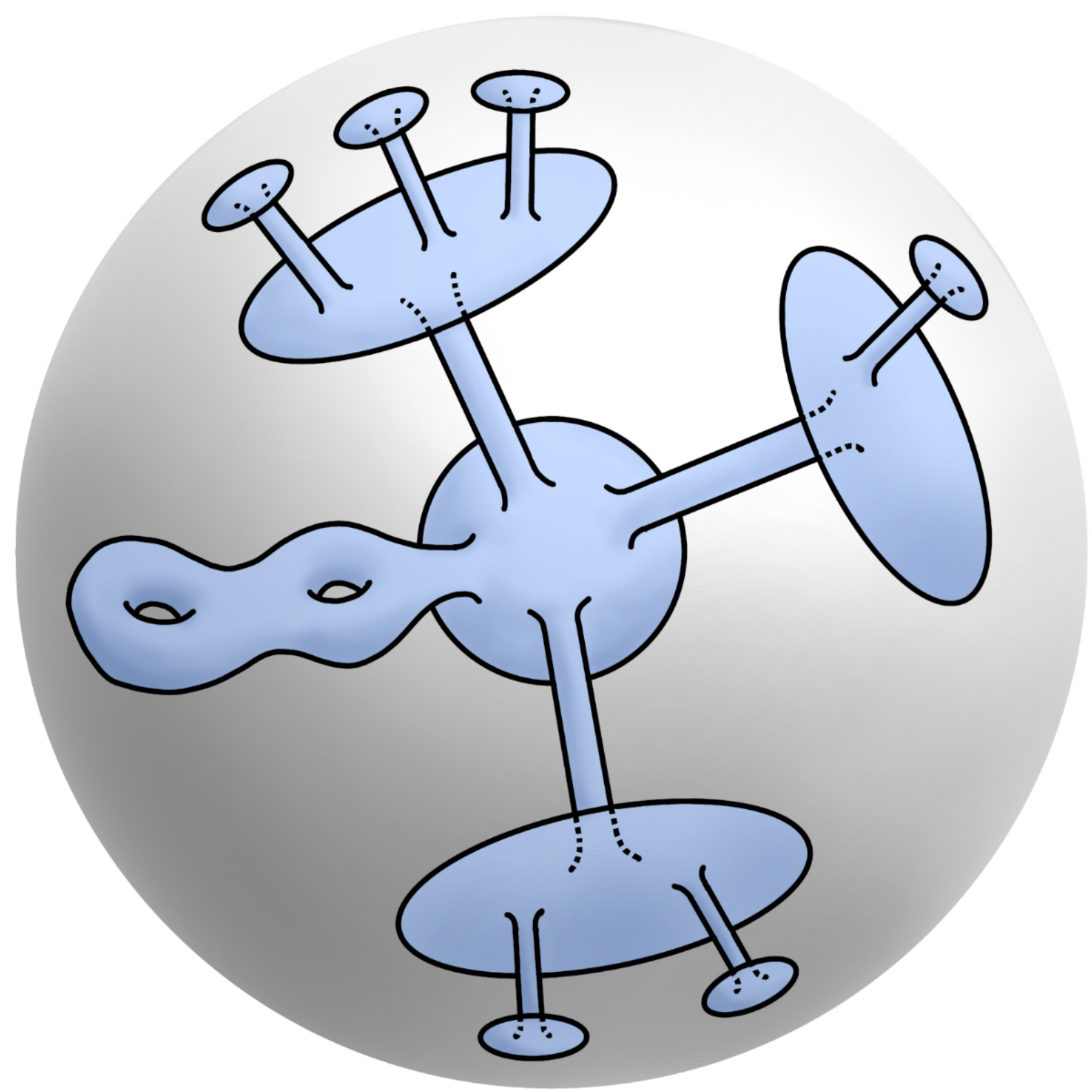}};
\end{tikzpicture}
\caption{Schematical representation of a model surface with prescribed genus and boundary graph.}
\label{fig:ModelSurface}
\end{figure}

Choose a node $v_0$ to be the root of the tree, and denote by $v_1,\ldots,v_k$ its direct children for some $k\in\N$.
Let $S$ be a sphere centered in the origin and with radius $1/8$. Moreover, take $k$ disjoint flat discs $D_1,\ldots,D_k\subset B^3$ with the same radius and with $\partial D_i = D_i\cap\partial B^3$ for $i=1,\ldots,k$.
Now, let $r_i$ be the radial segment connecting the center of $D_i$ with $S$ and let $C_i$ be a small tubular neighborhood of $r_i$ with $\partial C_i\subset D_i\cup S$. Assume that $C_1,\ldots,C_k$ have the same radius for $i=1,\ldots,k$.
Then, consider the surface $\Sigma_1$ obtained as the union of $S$, $D_1,\ldots,D_k$ and $C_1,\ldots,C_k$, after removing the small discs bounded by $\partial C_i$ on $D_i$ and $S$ for $i=1,\ldots,k$.

Now, consider the vertex $v_i$ for some $i=1,\ldots,k$. Let $k_i$ be the number of direct children of $v_i$, which are at distance $2$ from the root $v_0$. Consider $k_i$ disjoint flat discs in $B^3$ with boundary contained in the smallest of the components of $\partial B^3\setminus\partial D_i$.  Moreover, similarly as before, connected the center of each of these discs with $D_i$ with small cylindrical bridges. We choose the discs and the bridges of sufficiently small radius to ensure no self-intersections.
Repeat this procedure for each $i=1,\ldots,k$ and let $\Sigma_2$ be the surface obtained as the union of $\Sigma_1$ and all these new discs and cylindrical bridges.

We then repeat this procedure for the vertices in $T$ at distance $3$ from the root $v_0$, to get a surface $\Sigma_3$. And so on, we repeat the procedure until each nonroot vertex of $T$ corresponds to some disc connected with a cylindrical bridge to the disc corresponding to its mother in the tree. We let $\Sigma'$ be the surface that we obtain in this way.
Finally, let $\Sigma$ be the surface obtained by taking the connected sum of $\Sigma'$ with a standardly embedded genus $g$ surface. In particular, we can assume that the genus $g$ surface is attached to the sphere $S$ and it is small enough so that it does not intersect $\Sigma'$.
The surface $\Sigma$ is not smooth at the intersection between the discs and the bridges, but it can be smoothened out easily.
\end{definition}

\begin{proposition} \label{prop:ModelIsHeegaard}
The model surface in \cref{def:ModelSurface} is a strong Heegaard splitting.
\end{proposition}
\begin{proof}
By \cref{lem:HeegaardIrreducibility}, it suffices to show that if $\Sigma\subset B^3$ is a model surface of genus $g$ with boundary graph a tree $T$, then $\pi_1(\Sigma)\to\pi_1(C_i)$ is surjective for $i=1,2$, where $C_1$ and $C_2$ are the two connected components of $B^3\setminus\Sigma$. 

Now, take two points $p_1\in C_1$ and $p_2\in C_2$ such that $p_1,p_2$ are contained in a small tubular neighborhood of the sphere $S$. 
For every cylindrical bridge in $\Sigma$, consider a simple loop winding around the bridge and contained in $B^3\setminus \Sigma$. Note that each loop is either contained in $C_1$ or $C_2$. 
Then, we can connect each loop with a simple curve to $p_1$ if the loop is in $C_1$, or to $p_2$ if the loop is in $C_2$.
Let us denote by $\rho_1\subset C_1$ and $\rho_2 \subset C_2$ the two graphs obtained in this way.
Finally, add to $\rho_1$ one simple loop in $C_1$ (based in $p_1$) winding around each hole of the genus $g$ surface, and do the same with $\rho_2$.

Then, $C_1$ and $C_2$ can be retracted to the graphs $\rho_1$ and $\rho_2$, respectively. Therefore, the loops in the constructions of $\rho_1$ and $\rho_2$ are generators of the fundamental group of $C_1$ and $C_2$, respectively. 
Moreover, note that every loop in $\rho_1$ and $\rho_2$ can be isotoped to a loop in $\Sigma$. As a result, $\pi_1(\Sigma)\to\pi_1(C_i)$ is surjective for $i=1,2$. 
\end{proof}

\begin{theorem} \label{thm:HeegaardImpliesIsotopic}
Let $\Sigma$ be a (smooth, compact, connected, properly embedded) strong Heegaard splitting in a three-dimensional differentiable ball $M^3$. Then $\Sigma$ is \emph{unknotted}, meaning that it is smoothly isotopic in $M$ to the model surface from \cref{def:ModelSurface} with the same genus and the same boundary graph. 
In particular, every pair of strong Heegaard splittings with the same genus and the same boundary graph are smoothly isotopic.
\end{theorem}
\begin{proof}
Let $\Sigma'$ be the model surface defined in \cref{def:ModelSurface} with the same genus and same boundary graph as $\Sigma$. By \cref{prop:ModelIsHeegaard}, $\Sigma'$ is a strong Heegaard splitting of $M$ too. Therefore, Frohman in \cite{Frohman1992}*{Theorem~2.1} showed that there exists a homeomorphism $\tilde f\colon M\to M$ such that $\tilde f(\Sigma)\to\Sigma'$.

As a first step, we now want to smooth out $\tilde f$ to get a diffeomorphism.
Define the surfaces $X=\Sigma\cup\partial M$ and $X' = \Sigma'\cup\partial M$.
Thanks to \cite{Munkres1956}*{Chapter~V}, $\tilde f|_X\colon X\to X'$ is isotopic to a diffeomorphism $h_X\colon X\to X'$. This can be seen by applying Munkres' smoothing procedure first to $\tilde f|_{\partial M}$, sending $\partial \Sigma$ to $\partial \Sigma'$, and then to $\tilde f|_{\Sigma}$.

By \cite{EdwardsKirby1971}*{Corollary~1.2}, $\tilde f$ is isotopic to a homeomorphism $h\colon M\to M$ that coincides with $h_X$ on $X$. Moreover, we can assume that $h$ is a diffeomorphism from a tubular neighborhood $N(X)$ of $X$ to a tubular neighborhood $N(X')$ of $X'$. This can be achieved by interpolating $h_X$ in the tubular neighborhood $N(X)$.

Finally, the map $h\colon M\to M$ is isotopic to a diffeomorphism $f\colon M\to M$, that coincides with $h$ on $N(X)$, thanks to \cite{Cerf1959}*{Theorem~5 and Corollary~1} (which assumes the Smale's conjecture, now proven in Hatcher \cite{Hatcher1983} and Bamler--Kleiner \cite{BamlerKleiner2023}).

We are left to prove that $\Sigma$ and $\Sigma'$ are smoothly isotopic. 
Note that, by \cite[Appendix~1]{Hatcher1983}, the group of diffeomorphisms of $B^3$ that fix the boundary $\partial B^3=S^2$ is contractible. Moreover, the group of diffeomorphisms of $S^2$ has two connected components corresponding respectively to the orientation preserving and orientation reversing diffeomorphisms (see \cite{Smale1959}).
As a result, the diffeomorphism $f$ is smoothly isotopic to either the identity or to a reflection across a plane.
This implies that the surface $\Sigma$ is smoothly isotopic to either $\Sigma'$ or the reflection of $\Sigma'$ across a plane.
However, note that it is possible to perform the construction of $\Sigma'$ in \cref{def:ModelSurface} in such a way that the model surface is symmetric with respect to the reflection across a plane (by choosing the discs to be centered on such plane). This concludes the proof.
\end{proof}

\section{Free boundary minimal surfaces are strong Heegaard splittings}\label{sec:fbms}

In this section, we prove the unknottedness of free boundary minimal surfaces stated in \cref{thm:freeboundary}.
Thanks to the topological preliminaries in the previous section, the theorem is a consequence of the following result.

\begin{theorem}[cf. \cite{ChoeFraser2018}*{Lemma~3.1 and~Theorem~3.2}] \label{thm:FBMSHeegaard}
Let $\Sigma$ be a (smooth, compact, properly embedded) free boundary minimal surface in a three-dimensional Riemannian manifold $M^3$ fulfilling the hypotheses of \cref{thm:freeboundary}. Then $\Sigma$ is a strong Heegaard splitting of $\amb$.
\end{theorem}

\begin{proof}
Recall that $M$ is diffeomorphic to a three-dimensional ball by \cite{FraserLi2014}*{Theorem~2.11}.
Moreover, note that $\Sigma$ is connected by the Frankel property for free boundary minimal surfaces in $M$, proved in \cite{FraserLi2014}*{Lemma~2.4}. Therefore, by \cref{lem:HeegaardIrreducibility}, to show that $\Sigma$ is a strong Heegaard splitting we just need to prove that the induced maps $\pi_1(\Sigma)\to\pi_1(C_i)$ are surjective for $i=1,2$, where $C_1,C_2$ are defined as in \cref{def:HeegaardSplitting} since $\Sigma$ is a strong Heegaard splitting. Without loss of generality, inspired by \cite{Lawson1970unknot}, we let $i=1$ and consider the universal Riemannian cover $\Pi\colon \Tilde{C_1}\to C_1$. 
By \cref{rmk:HeegaardUniversalCover}, it is then sufficient to prove that $\Pi^{-1}(\Sigma)$ is connected.

Assume by contradiction that it is not, and denote by $\Sigma_1,\ldots,\Sigma_N$ the connected components of $\Pi^{-1}(\Sigma)$. Then $\Pi^{-1}(\partial{C_1})=\left(\cup_{k=1}^N \Sigma_k\right)\bigcup\Gamma$ for $\Gamma=\Pi^{-1}(\partial M\cap \partial C_1)$ (note that $\Gamma$ may be disconnected). Inspired by \cite{FraserLi2014}*{proof of Lemma~2.4}, we define $d_k(x)$ to be the distance from a point $x\in\Tilde{C_1}$ to $\Sigma_k$, and, for every $k\not=\ell=1,\ldots,N$, we consider 
\[
D\eqdef \inf_{k\not=\ell=1,\ldots,N}\inf_{x\in\tilde C_1} \big( d_k(x) + d_\ell(x) \big).
\]
Without loss of generality, we can assume that the first infimum is realized for $k=1$, $\ell=2$, as $k$ and $\ell$ varies in $1,\ldots,N$. Consider some minimizing sequence $(x_m)_{m\in\mathbb{N}}\subset \tilde C_1$ for $d_1+d_2$, namely $d_1(x_m)+d_2(x_m)\to \inf_{x\in\tilde C_1}(d_1(x)+d_2(x)) = D$ as $m\to\infty$. Fix some $\tilde{p}_0\in\tilde C_1$ and let
$$\mathcal{B}_0=\{\tilde{p}\in\tilde{C}_1:d(\tilde{p},\tilde{p}_0)\leq 2\diam(C_1)\}.$$ Then, for each $x_m$, there exists $x_m'\in\mathcal{B}_0$ such that $\Pi(x_m)=\Pi(x_m')$, i.e., there is an isometry of $\tilde C_1$ mapping $x_m$ to $x_m'$. 
Therefore, possibly renaming $x_m$ (to be $x_m'$), we can assume without loss of generality that the sequence $(x_m)_{m\in\N}$ is contained in $\mathcal{B}_0$. This ball is compact, so up to subsequence $(x_m)_{m\in\N}$ converges to some $x_0$ realizing the infimum, namely 
\[
d_1(x_0)+d_2(x_0) = \lim_{m\to\infty} \big(d_1(x_m)+d_2(x_m)\big) = \inf_{x\in\tilde C_1} \big(d_{1}(x)+d_2(x)\big) = D.
\]
Therefore, we can apply \cite{NaffZhu2024}*{Proposition~25} on $\tilde C_1$ with $N=\Gamma$, $f=0$, $\kappa=0$, $1/\alpha=0$ (note that $\tilde C_1$ has nonnegative Ricci curvature and $\Gamma$ is convex in $\tilde C_1$) and we obtain a contradiction. Indeed, observe that $\tilde C_1$ cannot be a product manifold $\Sigma_1\times[0,d]$ since $\Gamma$ is strictly convex in $\tilde C_1$.

So we have proved that $\Sigma_1, \Sigma_2$ cannot be disjoint, meaning that $\Pi^{-1}(\Sigma)$ is connected and $\pi_1(\Sigma)\to\pi_1(C_1)$ is surjective. Thus, this concludes the proof that $\Sigma$ is a strong Heegaard splitting.
\end{proof}

\begin{proof} [Proof of \cref{thm:freeboundary}]
The result follows from \cref{thm:FBMSHeegaard} together with \cref{thm:HeegaardImpliesIsotopic}, from the previous section.
\end{proof}

\section{Self-shrinkers are strong Heegaard splittings} \label{sec:selfshrinkers}

Let $\Sigma\subset\R^3$ be a self-shrinker of the mean curvature flow. Equivalently, $\Sigma$ is a minimal surface with respect to the Gaussian metric $e^{-\abs{x}^2/4}g_{\R^3}$ on $\R^3$.
Let us further assume that $\Sigma$ has finite topology. Then, by \cite{Wang2016}*{Theorem~1.1}, each end of $\Sigma$ is smoothly asymptotic to either a regular cone or a round cylinder. As a consequence, there exist $R_0>0$ such that for every $R\ge R_0$ the graphs $\Graph(\Sigma\cap \partial B_R(0),\partial B_R(0))$ are all isomorphic. 
Then, it makes sense to define the graph of $\Sigma$ at infinity, denoted $\Graph(\partial \Sigma,\infty)$, to be one of these isomorphic graphs.

As a result, to study the topology of a self-shrinker is sufficient to look at the portion of the self-shrinker inside a sufficiently large ball. 
In fact, we prove that the self-shrinker is a strong Heegaard splitting in every sufficiently large Euclidean ball.

\begin{theorem}
Let $\Sigma$ be a (smooth, complete, properly embedded) self-shrinker in $\R^3$ and take $M^3 = B_R(0)$ for some $R\ge2\sqrt{2}$. Then $\Sigma\cap M$ is a strong Heegaard splitting of $M$. 
\end{theorem}
\begin{proof}
First observe that $\Sigma\cap M$ is connected by the Frankel property for self-shrinkers, proved in \cite{ColdingMinicozzi2023}*{Corollary~0.3} and in \cite{NaffZhu2024}*{Theorem~3} with the explicit radius $2\sqrt{2}$.
Therefore, by \cref{lem:HeegaardIrreducibility}, we just need to prove that the induced maps $\pi_1(\Sigma)\to\pi_1(C_i)$ are surjective for $i=1,2$. 

Without loss of generality, let $i=1$ and consider the universal Riemannian cover $\Pi\colon (\Tilde{C_1},\tilde g_{0}) \to (C_1,g_{\R^3})$. By \cref{rmk:HeegaardUniversalCover}, $\pi_1(\Sigma)\to\pi_1(C_1)$ is surjective if and only if $\Pi^{-1}(\Sigma)$ is connected. 
Assume by contradiction that $\Pi^{-1}(\Sigma)$ is not connected, and denote by $\Sigma_1,\ldots,\Sigma_N$ the connected components of $\Pi^{-1}(\Sigma)$. We will now reach a contradiction by proving a Frankel property for the surfaces $\Sigma_1,\ldots,\Sigma_N$ on $\tilde C_1$, inspired by \cite{NaffZhu2024}*{Section~6.2 and Proposition~25}.

Consider the function $\tilde \rho\colon \Tilde{C_1}\to [0,\infty)$ given by $\tilde \rho(y) = R^2 - \abs{\Pi(y)}^2$, where we recall that $R$ is the radius of the ball $B_R(0)=M$. Moreover, let $\tilde \gamma\colon\tilde C_1\to[0,\infty)$ be defined as $\tilde\gamma(y) = \abs{\Pi(y)}^2/4$.
Let $\tilde f\colon \Tilde{C_1} \to (-\infty,\infty]$ be such that $e^{\tilde f} = \tilde\rho^{-2}$ and consider the metric $\tilde g_1 = e^{\tilde f}\tilde g_0$ on $\Tilde{C_1}$. Observe that $\tilde f = \infty$ on $\Gamma = \Pi^{-1}(\partial M \cap \partial C_1)$ (note that $\Pi^{-1}(\partial C_1) = (\cup_{k=1}^N\Sigma_k)\bigcup\Gamma$ and $\Gamma$ may be disconnected), so the idea to add this weight $\tilde f$ is to ``send the boundary $\Gamma$ at $\infty$''.

Note that what we are doing is to consider the same functions as in \cite{NaffZhu2024}*{Section~6.2} but lifted to the universal Riemannian cover of $\tilde C_1$. Indeed, with their notation of $\rho,\gamma, f$, we have that $\tilde\rho = \rho\circ \Pi$, $\tilde \gamma =\gamma\circ \Pi$ and $\tilde f = f\circ \Pi$.
In particular, as observed in \cite{NaffZhu2024}*{Section~6.2}, $\Sigma_1,\ldots,\Sigma_N$ are minimal surfaces in $(\tilde C_1,\tilde g_0)$ with respect to the weight $e^{-\tilde \gamma}$ (so we say that they are $\tilde\gamma$-minimal surfaces in $(\tilde C_1,\tilde g_0)$). Equivalently, $\Sigma_1,\ldots,\Sigma_N$ are minimal surfaces in $(\tilde C_1,\tilde g_1=e^{\tilde f}\tilde g_0)$ with respect to the weight $e^{-\tilde\gamma-\tilde f}$ (say $(\tilde\gamma+\tilde f)$-minimal surfaces in $(\tilde C_1,\tilde g_1)$).
Moreover, by the computation in \cite{NaffZhu2024}*{Section~6.2}, the manifold $(\tilde C_1,\tilde g_1)$ satisfies that the following $(-2)$-Bakry-Émery-Ricci curvature is nonnegative:
\[
\widetilde{\Ric}^{-2}_{\tilde \gamma+\tilde f} \ge 0,
\]
whenever $R\ge 2\sqrt{2}$.

Now, we define the function $d_k(x)$ to be the distance from a point $x\in\tilde C_1$ to $\Sigma_k$, for $k=1,\ldots,N$. Similarly to the proof of \cref{thm:FBMSHeegaard}, without loss of generality we get that
\[
\inf_{k\not=\ell =1,\ldots,N}\inf_{x\in\tilde C_1} \big(d_k(x) + d_\ell(x)\big)
\]
is achieved by $k=1,\ell=2$ and a point $x_0\in\tilde C_1\setminus\Gamma$.
As a result, we can apply \cite{NaffZhu2024}*{Proposition~25} on $(\tilde C_1,\tilde g_1, e^{-\tilde\gamma-\tilde f})$ with $\kappa =0$, $\alpha=-2$, $N=\tilde C_1$, and $\partial N=\emptyset$. Note that there is no free boundary in this application because all the boundary components are sent to infinity in the complete metric $\tilde g_1$ by the weight $e^{\tilde f}$. Thus, we obtain that $d_1+d_2$ is constant and $\Sigma_1,\Sigma_2$ are totally geodesic with respect to the metric $e^{-\tilde\gamma-\tilde f}\tilde g_1 = e^{-\tilde\gamma}\tilde g_0 = \Pi^*(e^{-\gamma} g_{\R^3})$. 
This implies that $\Sigma$ is totally geodesic in $B_R(0)$ with respect to the Gaussian metric $e^{-\gamma}g_{\R^3}$. By unique continuation (since the second fundamental form of a self-shrinker satisfies an elliptic equation, see e.g. \cite{ColdingMinicozzi2012}*{Lemma~10.8}), this implies that $\Sigma$ is totally geodesic in $\R^3$ with respect to the Gaussian metric.
We will now show that $\Sigma$ is either a plane passing through the origin or the sphere of radius $2$. In either cases, we conclude the proof, because they both imply that $\Pi^{-1}(\Sigma)$ is connected and therefore $\Sigma\cap M$ is a strong Heegaard splitting of $M$.

So, let us assume that $\Sigma$ is a (smooth, complete, properly embedded) surface in $\R^3$ that is totally geodesic with respect to the Gaussian metric. Then, the Euclidean second fundamental form of $\Sigma$ satisfies the equation
\[
A_{\R^3}^\Sigma = -\frac{x^\perp}4 g_{\R^3} = -\frac 14\sk{x,\nu}\nu\, g_{\R^3}
\]
for every $x\in\Sigma$. Here, $x^\perp$ is the projection of the position vector $x$ on the normal to $\Sigma$ at $x$, and $\nu$ is a choice of unit normal to $\Sigma$.

Let us distinguish two cases:
\begin{itemize}
\item Assume $\sk{x,\nu}$ vanishes somewhere, say at $x_0$. Then $T_{x_0}\Sigma$ is a plane passing through the origin, which is also a totally geodesic surface with respect to the Gaussian metric. Since totally geodesic surfaces are uniquely determined by their tangent plane in a point, $\Sigma$ coincides with $T_{x_0}\Sigma$.
\item Otherwise, assume that $\sk{x,\nu}$ is never zero. without loss of generality, we can assume that $\sk{x,\nu}>0$ on $\Sigma$. In particular, $\Sigma$ is a properly embedded mean convex self-shrinker in $\R^3$, this implies that $\Sigma$ is either a plane through the origin or the sphere of radius $2$ by \cite{ColdingMinicozzi2012}*{Theorem~0.17} (note that the polynomial volume growth assumption is satisfied here thanks to \cite{ChengZhou2013}*{Theorem~1.3}).
\end{itemize}
In both cases, we get that $\Sigma$ is either a plane through the origin or the sphere of radius $2$, which concludes the proof.
\end{proof}

\begin{proof}[Proof of \cref{thm:selfshrinkers}]
Let $\Sigma$ and $\Sigma'$ be two self-shrinkers in $\R^3$ with the same genus and with isomorphic graphs at infinity. 
Thanks to the structure theorem \cite{Wang2016}*{Theorem~1.1}, there exists $R\ge 2\sqrt{2}$ sufficiently large such that $\Sigma\setminus B_R(0)$ is smoothly isotopic in $\R^3\setminus B_R(0)$ to the cone over $\Sigma\cap \partial B_R(0)$, through an isotopy fixing $\partial B_R(0)$. Let $H_1\colon[0,1]\times\R^3\to\R^3$ be such isotopy of $\R^3$, which is the identity on $B_R(0)$. 
Similarly, we can assume $R$ to be also sufficiently large such that there exists a smooth isotopy $H_2\colon [0,1]\times\R^3\to\R^3$ that is the identity on $B_R(0)$ and maps $\Sigma'\setminus B_R(0)$ to the cone over $\Sigma'\cap \partial B_R(0)$.

If we further assume $R\ge 2\sqrt{2}$, thanks to \cref{thm:selfshrinkers} and \cref{thm:HeegaardImpliesIsotopic}, there exists a smooth isotopy $H_0\colon[0,1]\times \overline{B_R(0)}\to \overline{B_R(0)}$ of $B_R(0)$ mapping $\Sigma$ to $\Sigma'$. Let us extend $H_0$ to an isotopy $\tilde H_0\colon[0,1]\times\R^3\to\R^3$, by rescaling in $\R^3\setminus B_R(0)$. Namely, $\tilde H_0(t,x) = \frac{\abs{x}}{R} H_0(t,R\frac{x}{\abs{x}})$ for all $x\in \R^3\setminus B_R(0)$.
Then, the isotopy $H\colon[0,1]\times\R^3\to\R^3$ defined as $H(t,\cdot) = (H_2(t,\cdot))^{-1}\circ \tilde H_0(t,\cdot) \circ H_1(t,\cdot)$ maps $\Sigma$ to $\Sigma'$, which is the desired result.
\end{proof}

\bibliography{biblio}

\printaddress
\end{document}